\documentclass[a4paper,11pt]{article}
\usepackage{amsmath,amsthm,amsfonts,amssymb,enumitem,graphicx,calc,microtype,mathtools,thmtools,underscore,anyfontsize,lmodern,float}
\usepackage{latexsym}
\usepackage[dvipsnames]{xcolor}
\usepackage{todonotes}
\usepackage{wrapfig}
\usepackage{paralist}
\usepackage{enumitem}
\usepackage[english]{babel}
\usepackage[utf8]{inputenc}
\usepackage[T1]{fontenc}
\usepackage[margin=30mm]{geometry}
\renewcommand{\baselinestretch}{1.1}
\allowdisplaybreaks
\newcommand{\defn}[1]{\textcolor{Maroon}{\emph{#1}}}
\usepackage[numbers,sort&compress,longnamesfirst]{natbib}
\makeatletter
\def\NAT@spacechar{~}
\makeatother
\usepackage[pdftitle={Induced Subgraphs and Path Decompositions},
pdfauthor={Hickingbotham},
colorlinks =true,
linkcolor={black},
urlcolor={blue!80!black},
filecolor={blue!80!black},
citecolor={black},
anchorcolor=blue,
filecolor=blue, 
menucolor=blue]{hyperref}

\usepackage[noabbrev,capitalise]{cleveref}
\crefname{lem}{Lemma}{Lemmas}
\crefname{thm}{Theorem}{Theorems}
\crefname{cor}{Corollary}{Corollaries}
\crefname{prop}{Proposition}{Propositions}
\crefname{conj}{Conjecture}{Conjectures}
\crefname{open}{Open Problem}{Open Problems}
\crefname{obs}{Observation}{Observations}
\crefformat{equation}{(#2#1#3)}
\Crefformat{equation}{Equation #2(#1)#3}

\theoremstyle{plain}
\newtheorem{thm}{Theorem}
\newtheorem{lem}[thm]{Lemma}

\newtheorem{prop}[thm]{Proposition}

\theoremstyle{definition}

\DeclareMathOperator{\tw}{tw}

\DeclareMathOperator*{\pw}{pw}

\newcommand{\beginproof}{\emph{Proof.}}

\renewcommand{\leq}{\leqslant}

\renewcommand{\geq}{\geqslant}

\theoremstyle{definition}

\DeclareMathOperator{\dist}{dist}

\DeclareMathOperator*{\N}{\mathbb{N}}

\DeclareMathOperator*{\G}{\mathcal{G}}

\DeclareMathOperator*{\W}{\mathcal{W}}

\presetkeys%
{todonotes}%
{inline}{}

\begin{document}
	
	\title{\bf\Large Induced Subgraphs and Path Decompositions}
	\author{%
		Robert Hickingbotham\thanks{School of Mathematics, Monash University, Melbourne, Australia (\texttt{robert.hickingbotham@monash.edu}). Research supported by an Australian Government Research Training Program Scholarship.}
% 		\quad
% 		David R. Wood\thanks{School of Mathematics, Monash University, Melbourne, Australia (\texttt{david.wood@monash.edu}). Research supported by the Australian Research Council.}
	}
	
	\date{\normalsize\today}
	\maketitle
	
\begin{abstract}
A graph $H$ is an induced subgraph of a graph $G$ if a graph isomorphic to $H$ can be obtained from $G$ by deleting vertices. Recently, there has been significant interest in understanding the unavoidable induced subgraphs for graphs of large treewidth. Motivated by this work, we consider the analogous problem for pathwidth: what are the unavoidable induced subgraphs for graphs of large pathwidth? While resolving this question in the general setting looks challenging, we prove various results for sparse graphs. In particular, we show that every graph with bounded maximum degree and sufficiently large pathwidth contains a subdivision of a large complete binary tree or the line graph of a subdivision of a large complete binary tree as an induced subgraph. Similarly, we show that every graph excluding a fixed minor and with sufficiently large pathwidth contains a subdivision of a large complete binary tree or the line graph of a subdivision of a large complete binary tree as an induced subgraph. Finally, we present a characterisation for when a hereditary class defined by a finite set of forbidden induced subgraphs has bounded pathwidth. 
\end{abstract}

\section{Introduction}	
Pathwidth and treewidth are fundamental parameters in structural and algorithmic graph theory \citep{Reed03,HW2017tied,bodlaender1998partial}. A paramount theme has been to understand the substructures of graphs with large pathwidth or large treewidth. Under the graph minor and subgraph relations, these substructure are well understood. For pathwidth, the excluded forest minor theorem \cite{robertson1983graph} implies that graphs with sufficiently large pathwidth contains a subdivision of a large complete binary tree as a subgraph. For treewidth, the excluded grid minor theorem \cite{robertson1986planar} implies that graphs with sufficiently large treewidth contains a subdivision of a large wall as a subgraph.

In recent years, there has been substantial interest in understanding the substructures for graphs with large treewidth with respect to induced subgraphs \cite{LR2022treewidth,AAKST2021evenhole,korhonen2022,PSTT2021theta,ST2021wheels,ACV2022inducedI,ACDHRSV2021inducedII,ACMHS2021inducedIII,ACHS2022inducedIV,ACHS2022inducedV}. For example, it was recently shown that graphs with bounded degree and sufficiently large treewidth contains a subdivision of a large wall or the line graph of a subdivision of a large wall as an induced subgraph \cite{korhonen2022}. Motivated by this work, we consider the analogous problem for pathwidth: what are the unavoidable induced subgraphs for graphs with large pathwidth? Due to the excluded forest minor theorem, obvious candidates are subdivisions of large complete binary trees and line graphs\footnote{The \defn{line graph $L(G)$} of a graph $G$ has $V(L(G))=E(G)$ where two vertices in $L(G)$ are adjacent if they are incident to a common vertex in $G$.} of subdivisions of large complete binary trees. While determining the other candidates in the general setting looks challenging (see \cref{SectionGeneral} for a discussion on this), we show that these graphs suffice in the bounded degree setting (\cref{SectionWattletoSubgraph}) as well as for $K_n$-minor-free graphs (\cref{SectionMinorFreetoInduced}). Let \defn{$T_k$} denote the complete binary tree of height $k$. 

\begin{thm}\label{MainTheoremPW}
There is a function $f$ such that every graph $G$ with maximum degree $\Delta$ and pathwidth at least $f(k,\Delta)$ contains a subdivision of $T_k$ or the line graph of a subdivision of $T_k$ as an induced subgraph.
\end{thm}

\begin{thm}\label{MainTheoremPWMinor}
For every fixed $n\in \N$, there is a function $f$ such that every $K_n$-minor-free graph $G$ with pathwidth at least $f(k)$ contains a subdivision of $T_k$ or the line graph of a subdivision of $T_k$ as an induced subgraph.
\end{thm}

In addition, we characterise when a hereditary graph class defined by a finite set of forbidden induced subgraphs has bounded pathwidth. Let \defn{$\G_S$} be the class of graphs that contain no graph in S as an induced subgraph. We call $K_{1,3}$ a \defn{claw}. A \defn{fork} is a subdivision of a claw and a \defn{semi-fork} is the line graph of a fork. A \defn{tripod} is a forest where each component is a fork and a \defn{semi-tripod} is a graph where each component is a semi-fork. We prove the following characterisation for when $\G_S$ has bounded pathwidth.

\begin{restatable}{thm}{FiniteSetPW}
\label{thm:FiniteSetPW}
   $\G_S$ has bounded pathwidth if and only if $S$ includes a complete graph, a complete bipartite graph, a tripod and a semi-tripod.
\end{restatable}

\section{Preliminaries} 
For undefined terms and notations, see the textbook by \citet{diestel2017graphtheory}. A \defn{class} of graphs $\G$ is a set of graphs that is closed under isomorphism. $\G$ is \defn{hereditary} if it is closed under induced subgraphs. For a graph $G$ and vertex $v\in V(G)$, let \defn{$N(v)$} denote the set of all vertices in $V(G)$ adjacent to $v$ and let \defn{$N[v]$} denote $N(v)\cup \{v\}$.

Let $G$ and $H$ be graphs and let $r\geq 0$ be an integer. $H$ is a \defn{minor} of $G$ if a graph isomorphic to $H$ can be obtained from $G$ by vertex deletion, edge deletion, and edge contraction. If $H$ is isomorphic to a graph obtained from $G$ by vertex deletion and edge contraction, then $H$ is an \defn{induced minor} of $G$. A \defn{(minor) model} $(X_v:v\in V(H))$ of $H$ in $G$ is a collection of subsets of $V(G)$ such that (i) $G[X_v]$ is a connected subgraph of $G$; (ii) $X_u\cap X_v=\emptyset$ for all distinct $u,v\in V(H)$; and (iii) $G[X_u\cup X_v]$ is connected for every edge $uv \in E(H)$. If, in addition, $G[X_u\cup X_v]$ is disconnected whenever $uv\not \in E(H)$, then $(X_v:v\in V(H))$ is an \defn{induced minor model} of $H$. It is folklore that $H$ is a minor of $G$ if and only if $G$ contains a model of $H$ and that $H$ is an induced minor of $G$ if and only if $G$ contains an induced minor model of $H$. If there exists a model $(X_v:v\in V(H))$ of $H$ in $G$ such that $G[X_v]$ has radius at most $r$ for all $v \in V(H)$, then $H$ is an \defn{$r$-shallow minor} of $G$. A graph $G'$ is a \defn{subdivision} of $G$ if $G'$ can be obtained from $G$ by replacing each edge $uv$ of $G$ by a path with end-vertices $u$ and $v$ whose internal vertices are new vertices private to that path. We call the vertices $V(G)\subseteq V(G')$ the \defn{original vertices} in $G'$.

For a tree~$T$, a \defn{$T$-decomposition} of a graph~$G$ is a collection~${\W = (W_x \colon x \in V(T))}$ of subsets of~${V(G)}$ indexed by the nodes of~$T$ such that
(i) for every edge~${vw \in E(G)}$, there exists a node~${x \in V(T)}$ with~${v,w \in W_x}$; and 
(ii) for every vertex~${v \in V(G)}$, the set~${\{ x \in V(T) \colon v \in W_x \}}$ induces a (connected) subtree of~$T$. 
The \defn{width} of~$\W$ is $\max\{ |W_x| \colon x \in V(T) \}-1$. 
A \defn{tree-decomposition} is a $T$-decomposition for any tree~$T$. A \defn{path-decomposition} is a $P$-decomposition for any path~$P$. 
The \defn{treewidth}~$\tw(G)$ of a graph~$G$ is the minimum width of a tree-decomposition of~$G$ and the \defn{pathwidth} $\pw(G)$ of a graph~$G$ is the minimum width of a path-decomposition of $G$. Clearly $\tw(G)\leq \pw(G)$ for all graphs $G$.
Treewidth and pathwidth are the standard measures of how similar a graph is to a tree and to a path respectively. A fundamental result for pathwidth is the excluded forest minor theorem.
\begin{thm}[\cite{robertson1983graph,bienstock1991quickly}]\label{ExcludeForestPW}
    For every forest $F$, every graph with pathwidth at least $|V(F)|-1$ contains $F$ as a minor.
\end{thm}
Note that \cref{ExcludeForestPW} implies that every graphs with pathwidth at least $|V(T_k)|-1$ contains a subdivision of $T_k$ as a subgraph since $T_k$ have maximum degree $3$. 
\section{Results}

\subsection{From Induced Minors to Induced Subgraphs}
To prove \cref{MainTheoremPW,MainTheoremPWMinor}, we first construct an induced minor of a large complete binary tree in our graph. We use the following lemma to then find our desired induced subgraphs.

\begin{lem}\label{InducedMinortoInducedSubgraph}
Every graph $G$ that contains $T_{f(k)}$ as an induced minor contains a subdivision of $T_k$ or the line graph of a subdivision of $T_k$ as an induced subgraph where $f(k)=2k^2+10k$.
\end{lem}

The rest of this subsection is dedicated to proving \cref{InducedMinortoInducedSubgraph}. A \defn{net-graph} is a semi-fork obtained from a triangle by appending disjoint paths of length $1$ at each vertex of the triangle. For a graph $G$ and a vertex $v\in V(G)$ with degree $3$ and neighbours $a,b,c$, a \defn{net graph replacement} at $v$ is the graph $G'$, with $V(G')=V(G)\setminus\{v\}\cup \{x,y,z\}$ where $x,y,z$ are new vertices and $E(H)=E(G-v)\cup \{xy,yz,zx,xa,yb,zc\}$. A \defn{wattle $\tilde{T}_{k}$} is obtained from a subdivision of $T_k$ by picking a (possibly empty) subset $X$ of the degree $3$-vertices and performing net graph replacements at each vertex in $X$. The following lemma from \citet{AAKST2021evenhole} will be useful in constructing an induced wattle $\tilde{T}_{k}$ from a large induced minor.

\begin{lem}[\cite{AAKST2021evenhole}]\label{CleaningInducedMinor}
 Let $G$ be a connected graph whose vertex set is partitioned into connected sets $A,B,C,\{a\},\{b\},\{c\}$ and $S$. Suppose that every edge of $G$ has either both ends in one of the sets, or is from $\{a\}$ to $A$, from $\{b\}$ to $B$, from $\{c\}$ to $C$, or from $S$ to $A\cup B\cup C$. Then $a$, $b$, $c$ are the degree one vertices of some induced fork or semi-fork in $G$.
\end{lem}

For a rooted tree $(T,r)$, a vertex $u\in V(T)$ is an \defn{ancestor} of
$v\in V(T)$ (and $v$ is a \defn{descendant} of $u$) if $u$ is a vertex on the $(v,r)$-path in $T$. If $p_v\in V(T)$ is the neighbour of $v$ on the $(v,r)$-path in $T$, then $p_v$ is the \defn{parent} of $v$ (and $v$ is a \defn{child} of $p_v$). For a rooted complete binary tree $(T_k,r)$, every vertex $u\in V(T_k)$ with degree at least $2$ has a \defn{left child $\ell_v$} and a \defn{right child $c_v$}. A vertex $v\in V(T_k)$ is a \defn{left (right) descendent} of $u$ if it is a descendant of the left (right) child of $u$.

\begin{lem}\label{MinorToWattle}
 If a graph $G$ contains $T_{4k}$ as an induced minor, then $G$ contains a wattle $\tilde{T}_{k}$ as an induced subgraph.
\end{lem}

\beginproof\ 
    Let $(X_v\colon v\in V(T_{4k}))$ be an induced minor model of $(T_{4k},r)$ in $G$. Let $V_0,V_1,\dots,V_{4k}$ be a \textsc{bfs}-layering of $T_{4k}$ where $V_0=\{r\}$.\footnote{$V_0,V_1,\dots,$ is a \defn{\textsc{bfs}-layering} of a graph $G$ if $V_0 = \{r\}$ for some $r\in V(G)$ and $V_i=\{v\in V(G):\dist_G(v,r)=i\}$ for all $i\geq 1$.} For a vertex $v\in V(T_{4k})$ with degree $3$, let $P_v:=(w_{v,0},w_{v,1}\dots,w_{v,m})$ be a vertex-minimal path in $G[X_v]$ such that $N_G(w_{v,0})\cap X_{p_v}$ and $N_G(w_{v,m})\cap X_{\ell_v}$ are non-empty. By minimality, $w_{v,0}$ is the only vertex in $V(P_v)$ adjacent to vertices in $X_{p_v}$ and $w_{v,m}$ is the only vertex in $V(P_v)$ adjacent to vertices in $X_{\ell_v}$. 
    
     \begin{wrapfigure}{R}{0.35\textwidth}
        \centering
    	\includegraphics[width=0.95\linewidth]{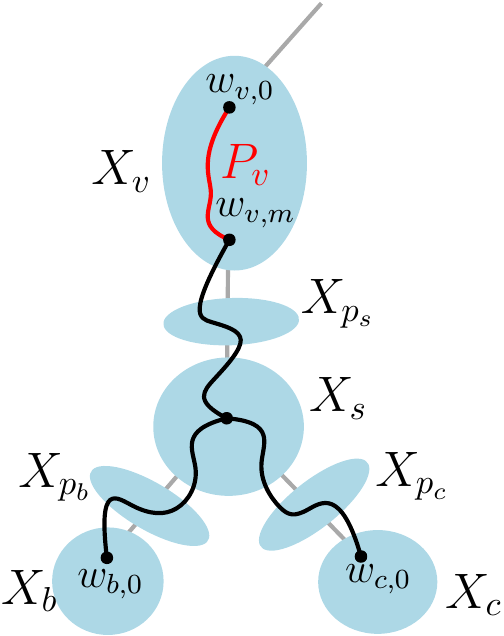}
    	\caption{Extending the wattle $\tilde{T}_k$.}
    	\label{fig:WattleTk}
     \end{wrapfigure}
    
    We prove the following claim by induction on $k\geq 0$: If a graph $G$ contains $T_{4k}$ as an induced minor, then $G$ contains a wattle $\tilde{T}_{k}$ as an induced subgraph whose leaves are contained in $\{w_{v,0}\colon v \in V_{4k}\}$ and every vertex in $V(\tilde{T}_{k})\cap (\bigcup(V(X_v)\colon v\in V_{4k}))$ is a leaf.
    
    For $k=0$, the claim holds trivially by letting $V(\tilde{T}_{k})$ be an arbitrary vertex in $X_r$. For $k=1$, let $a,b\in V_4$ respectively be a left and right descendent of $r$. By taking $V(\tilde{T}_{k})$ to be a vertex minimal $(w_{a,0},w_{b,0})$-path such that $V(\tilde{T}_{k})\cap (\bigcup(V(X_v)\colon v\in V_{4}))=\{w_{a,0},w_{b,0}\}$, we are done.
    
    Now suppose the claim holds for $k-1$. By induction, $T_{4k-4}$ contains a wattle $\tilde{T}_{k-1}$ as an induced subgraph whose leaves are contained in $\{w_{v,0}\colon v \in V_{4k}\}$ and every vertex in $V(\tilde{T}_{k})\cap  (\bigcup(V(X_v)\colon v\in V_{4k}))$ is a leaf. Now consider a leaf $w_{v,0}\in V(\tilde{T}_{k-1})$. First append the path $P_v$ to $\tilde{T}_{k-1}$. Let $s\in V_{4k-2}$ be a left descendants of $v$ and let $b,c\in V_{4k}$ respectively be left and right descendants of $s$. Let $A=X_{p_{s}}$, $B=X_{p_{b}}$, $C=X_{p_{c}}$ and $S=X_{s}$. Since $(X_v\colon v\in V(T_{4k}))$ is an induced minor, we may apply \cref{CleaningInducedMinor} on $G[\{w_{v,m},w_{b,0},w_{c,0}\}\cup A\cup B\cup C \cup S]$ to obtain an induced fork or semi-fork with end points $w_{v,m},w_{b,0},w_{c,0}$. Add this induced fork or semi-fork to $\tilde{T}_{k-1}$ and repeat for all leaves in $\tilde{T}_{k-1}$ to obtain an induced wattle $\tilde{T}_{k}$ that satisfies the induction hypothesis (see \cref{fig:WattleTk}).\qed
 
For a rooted tree $(T,r)$, we say that a rooted subtree $(T',r')$ is \defn{vertical} if $r'$ is an ancestor (with respect to $(T,r)$) of every vertex in $V(T')$. We use the following lemma to clean up our wattle $\tilde{T}_{k}$. 
 
 \begin{lem}\label{MonochromaticCBT}
 For every red-blue colouring of $(T_{f(k)},r)$, there exists a subdivision of a vertical $(T_k,r)$ whose original vertices are monochromatic where $f(k)=(k^2+5k)/2$.
\end{lem}

\begin{proof}
    We proceed by induction on $k$. For $k=0$, the claim is trivial. Now suppose the claim hold for $k-1$. Let $(V_0,\dots,V_{f(k)})$ be a \textsc{bfs}--layering of $(T_{f(k)},r)$ where $V_0=\{r\}$. Observe that $f(k)=f(k-1)+k+2$. Since each component of $T_{f(k)}[V_{k+2}\cup \dots\cup V_{f(k)}]$ is isomorphic to $T_{f(k-1)}$, by induction, each of these components contain a subdivision of a vertical $T_{k-1}$ whose original vertices are monochromatic. We say that such a subdivided tree is \defn{left} (\defn{right}) if its vertices are left (right) descendants of $r$ and that it is \defn{red} (\defn{blue}) if its original vertices are red (blue).
    
    Without loss of generality, assume that $r$ is coloured red. Now if there is a red left tree $(T_{k-1}^{(1)},r^{(1)})$ and a red right tree $(T_{k-1}^{(2)},r^{(2)})$, then together with the $(r^{(1)},r^{(2)})$-path in $T$ (which goes through $r$), we have a vertical red $(T_k,r)$. So assume without loss of generality that all the left trees are blue. If there is a blue vertex $r'\in V_1\cup \dots \cup V_{k+1}$ that is a left descendent of $r$, then we can constructed a vertical blue $(T_k,r')$ with two of the blue left trees. Otherwise all the left descendants of $r$ in $V_1\cup \dots \cup V_{k+1}$ are red which induces a vertical red $(T_k,r')$ where $r'$ is the left child of $r$.
\end{proof}

\begin{lem}\label{WattletoSubgraph}
Every wattle $\tilde{T}_{f(k)}$ contains a subdivision of $T_{k}$ or the line graph of a subdivision of $T_k$ as an induced subgraph where $f(k)=(k^2+5k)/2$.
\end{lem}
\begin{proof}
    Let ${T}_{f(k)}'$ be an auxiliary copy of $T_{f(k)}$ obtained from the wattle $\tilde{T}_{f(k)}$ by first contracting each triangle into a red vertex then contracting the subdivided paths and colouring the remaining vertices blue. By \cref{MonochromaticCBT}, ${T}_{f(k)}'$ contains a subdivision ${T}_{k}'$ of $T_k$ whose original vertices are monochromatic. If the original vertices of ${T}_{k}'$ are red, then the wattle $\tilde{T}_{f(k)}$ contains the line graph of a subdivision of $T_k$ as an induced subgraph (where each triangle in the line graph corresponds to an original red vertex in ${T}_{k}'$). Otherwise the original vertices of ${T}_{k}'$ are blue and thus the wattle $\tilde{T}_{f(k)}$ contains a subdivision of $T_{k}$ as an induced subgraph (where the original vertices correspond to the original blue vertices in ${T}_{k}'$).
\end{proof}

\cref{InducedMinortoInducedSubgraph} immediately follows from \cref{MinorToWattle,WattletoSubgraph}.

\subsection{From Bounded Degree to Induced Minors}\label{SectionWattletoSubgraph}
We now show that graphs with bounded degree and sufficiently large pathwidth contains a large complete binary tree as an induced minor. \cref{MainTheoremPW} immediately follows from the next theorem together with \cref{InducedMinortoInducedSubgraph}.

\begin{thm}\label{MainTheoremInduced}
There is a function $f$ such that every graph with maximum degree $\Delta$ and pathwidth at least $f(k,\Delta)$ contains $T_k$ as an induced minor.
\end{thm}

To prove \cref{MainTheoremInduced}, we use sparsifable graphs which is a new technique introduced by \citet{korhonen2022}. For a graph $G$, a vertex $v\in V(G)$ is \defn{sparsifable} if it satisfies one of the following conditions:
\begin{compactenum}
\item $v$ has degree at most $2$;
\item $v$ has degree $3$ and all of its neighbours have degree at most $2$;
\item $v$ has degree $3$, one of its neighbours has degree at most $2$, and the two other neighbours form a triangle with $v$.
\end{compactenum}
We say that $G$ is \defn{sparsifable} if all of its vertices are sparsifable. Such graphs are useful since minors and induced minors are roughly equivalent in this setting. More precisely, \citet{korhonen2022} showed that if a sparsifable graph $G$ contains a graph $H$ with minimum degree at least $3$ as a minor, then $G$ contains $H$ as an induced minor. We prove a slightly stronger version of this result that relaxes the minimum degree condition for $H$. 

\begin{lem}\label{SparsifableInducedMinors}
    Let $H$ and $H^+$ be graphs such that $H$ is an induced subgraph of $H^+$ and each vertex in $V(H)$ has degree at least $3$ in $H^+$. If a sparsifable graph $G$ contains $H^+$ as a minor, then $G$ contains $H$ as an induced minor.
\end{lem}

\begin{proof}
    For a model $(X_v:v\in V(H^+))$ of $H^+$ in $G$, we say that an edge $ab\in E(G)$ is \defn{$H$-violating} if there are vertices $u,v\in V(H)$ such that $a\in X_u$, $b\in X_v$ and $uv\not\in E(H)$. Choose $(X_v:v\in V(H^+))$ such that the number of $H$-violating edges is minimised. We claim that there are no $H$-violating edges which implies that $(X_v:v\in V(H))$ is an induced minor of $H$ in $G$.
    
    For the sake of contradiction, suppose $ab\in V(G)$ is an $H$-violating edge with $a\in X_u$ and $b\in X_v$. Since $a$ has degree at most $3$ in $G$, $u$ has degree at least $3$ in $H^+$ and $uv\not \in E(H^+)$, it follows that $a$ has a neighbour in $X_u$. Likewise, $b$ has a neighbour in $X_v$. Now if $a$ has degree $2$ in $G$, then $G[X_u\setminus \{a\}]$ is connected and so we may replace $X_u$ by $X_u\setminus \{a\}$ to obtain a model of $H^+$ with strictly less $H$-violating edges, a contradiction. Thus $a$ must have degree $3$ in $G$ and $b$ likewise must also have degree $3$ in $G$.
    
    Now since $G$ is sparsifable and $a$ and $b$ are adjacent with degree $3$, they have a common neighbour $c$ in $G$. If $c\not \in \bigcup(X_w:w\in N_{H^+}[u])$, then $G[X_u\setminus \{a\}]$ is connected and so we may replace $X_u$ by $X_u\setminus \{a\}$ to obtain a model of $H^+$ with strictly less $H$-violating edges, a contradiction. By symmetry, a contradiction also occurs if $c \not \in \bigcup(X_w:w\in N_{H^+}[v])$. Since $uv\not \in E(H^+)$ it remains to consider the case when there exists $w\in V(H^+)\setminus\{u,v\}$ such that $c\in X_w$ and $uw,vw\in E(H^+)$. Since $w \not \in \{u,v\}$, it follows that $G[X_u\setminus \{a\}]$ and $G[X_v\setminus \{b\}]$ are connected. As such, by replacing $X_u$ by $X_u\setminus \{a\}$, $X_v$ by $X_v\setminus \{b\}$ and $X_w$ by $X_w\cup \{a,b\}$, we obtain a model of $H^+$ with strictly less $H$-violating edges, a contradiction.
\end{proof}

A \defn{distance-5 independent set} $\mathcal{I}\subseteq V(G)$ in a graph $G$ is a set of vertices such that the distance between any pair of vertices in $\mathcal{I}$ is at least $5$. For a distance-5 independent set $\mathcal{I}$ in a graph $G$, let \defn{${\mathcal{I}(G)}$} be the $2$-shallow minor of $G$ obtained by contracting each of the balls of radius 2 that are centered at vertices in $\mathcal{I}$ with corresponding model $(X_v:v\in V(\mathcal{I}(G))$. Observe that if $G$ has maximum degree $\Delta$, then $|X_v|\leq \Delta^2+1$ for all $v \in V(\mathcal{I}(G))$ and so $\pw(\mathcal{I}(G))\geq (\pw(G)+1)/(\Delta^2+1)-1$ (see Theorem~21 in \cite{BGHK1995approximating} for an implicit proof). We use the following lemma implicitly proved by \citet{korhonen2022}.

\begin{lem}[\cite{korhonen2022}]\label{SparsifableDegree}
    Let $G$ be a graph and $\mathcal{I}\subseteq V(G)$ be a distance $5$-independent set. Then for any subgraph $H'$ of ${\mathcal{I}(G)}$ with maximum degree 3, there exists an induced subgraph $G[S]$ of $G$ such that $G[S]$ contains $H'$ as a minor and every vertex in $I\cap S$ is sparsifable in $G[S]$.
\end{lem}

\begin{lem}\label{pwdeg3}
 There exists a constant $\delta$ such that every graph $G$ with $\pw(G)\geq 2k^2\log^{\delta}(k)$ contains a subgraph $H$ with maximum degree $3$ and $\pw(H)\geq k$.
\end{lem}

\begin{proof}
    By a result of \citet{chekuri2014degree}, if $\tw(G)\geq k\log^{\delta}k$ for some constant $\delta$, then $G$ contains a subgraph $H$ with maximum degree $3$ and treewidth at least $k$. Since pathwidth is bounded from below by treewidth, we are done. So assume that $\tw(G)< \log^{\delta}k$. In which case, by a result of \citet{groenland2021approximating}, $G$ contains a subdivision of the complete binary tree $T_{2k}$ as a subgraph which has the desired pathwidth \cite{scheffler1989die}.
\end{proof}

Let \defn{$T_k^+$} be the tree obtained from $T_{k+1}$ by adding a leaf vertex adjacent to the root. Observe that $T_k$ is an induced subtree of $T_k^+$ where each vertex in $V(T_k)$ has degree at least $3$ in $T_k^+$. We are now ready to prove \cref{MainTheoremInduced}.

\begin{proof}[Proof of \cref{MainTheoremInduced}]
    Let $g(\Delta^4+1):=2^{k+2}-1$, $g(i):= \big(2(g(i+1))^2\log^{\delta}(g(i+1))+1\big)(\Delta^2+1)-1$ and $f(k,\Delta):=g(0)$ where $\delta$ is from \cref{pwdeg3}. Let $G$ be a graph with maximum degree $\Delta$ and $\pw(G)\geq f(k,\Delta)$. We first construct a sparsifable induced subgraph of $G$ with large pathwidth. Using a greedy algorithm, partition $V(G)$ into $\Delta^4+1$ distance-5 sets $\mathcal{I}_0,\dots,\mathcal{I}_{\Delta^4}$. Initialise $i:=0$ and $G_i:=G$. We construct $G_{i+1}$ as an induced subgraph of $G_i$ such that every vertex in $\mathcal{I}_i\cap V(G_{i+1})$ is sparsifable in $G_{i+1}$ and $\pw(G_{i+1})\geq g(i+1)$.
    Since $G_i$ has maximum degree $\Delta$, $\pw(\mathcal{I}_i(G_i))\geq (\pw(G_i)+1)/(\Delta^2+1)-1\geq 2(g(i+1))^2\log^{\delta}(g(i+1))$. By \cref{pwdeg3}, $\mathcal{I}_i(G_i)$ contains a subgraph $H_i$ with maximum degree $3$ where $\pw(H_i)\geq g(i+1).$ By \cref{SparsifableDegree}, there exists an induced subgraph $G_i[S_i]$ of $G_i$ that contains $H_i$ as a minor and every vertex in $\mathcal{I}_i\cap S_i$ is sparsifable in $G_i[S_i]$. As pathwidth is closed under minors, $\pw(G_i[S_i])\geq \pw(H_i)\geq g(i+1)$. Set $G_{i+1}:=G_i[S_i]$.
    
    Now consider $\tilde{G}:=G_{d^4+1}$. By the above procedure, $\tilde{G}$ is a sparsifable induced subgraph of $G$ with pathwidth at least $g(\Delta^4+1)=2^{k+2}-1=|V(T_k^+)|-1$. By \cref{ExcludeForestPW}, $\tilde{G}$ contains $T_k^+$ as a minor. Therefore, by \cref{SparsifableInducedMinors}, $\tilde{G}$ contains $T_k$ as an induced minor and hence $G$ contains $T_k$ as an induced minor.
\end{proof}

\subsection{From Minor-Free to Induced Minors}\label{SectionMinorFreetoInduced}
Let \defn{$\mathcal{T}_k$} be the rooted tree of height $k$ in which every non-leaf node has $k$ children and every path from the root to a leaf has $k$ edges. We prove the following result for $K_n$-minor-free graphs.

\begin{thm}\label{MainHminor}
There is a function $f$ such that every $K_n$-minor-free graph $G$ with pathwidth at least $f(k,n)$ contains $\mathcal{T}_k$ as an induced minor. 
\end{thm}

Since $T_k$ is an induced subgraph of $\mathcal{T}_k$,
\cref{InducedMinortoInducedSubgraph,MainHminor} imply \cref{MainTheoremPWMinor}. \cref{MainHminor} quickly follows from the recent Ramsey-type result by \citet{AL2022degeneracy}. 

\begin{lem}[\cite{AL2022degeneracy}]\label{NtreeRamsey}
There is a function $g$ such that any graph that contains $\mathcal{T}_{g(n)}$ as a subgraph contains $K_n$, $K_{n,n}$ or $\mathcal{T}_{n}$ as an induced subgraph.
\end{lem}

\begin{proof}[Proof of \cref{MainHminor}]
Let $f(k,n):=|V(\mathcal{T}_{g(\max\{k,n\})})|-1$ where $g$ is from \cref{NtreeRamsey}.\footnote{Note that $|V(\mathcal{T}_k)|=(k^{k+1}-1)/(k-1)$.} By \cref{ExcludeForestPW}, $G$ contains a minor model $(X_v:v\in V(\mathcal{T}_{g(\max\{k,n\})}))$ of $\mathcal{T}_{g(\max\{k,n\})}$. Let $G'$ be the induced minor of $G$ obtained from contracting each of the $X_v$'s. Since $G'$ contains $\mathcal{T}_{g(\max\{k,n\})}$ as a subgraph, it follows by \cref{NtreeRamsey} that $G'$ contains $K_{n}$, $K_{n,n}$ or $\mathcal{T}_{k}$ as an induced subgraph. Since $G$ excludes $K_n$ as a minor, $G'$ does not contain $K_n$ or $K_{n,n}$ as subgraphs. Hence $G'$ contains $\mathcal{T}_{k}$ as an induced subgraph and thus $G$ contains $\mathcal{T}_{k}$ as an induced minor. 
\end{proof}

\subsection{Finitely Many Forbidden Induced Subgraphs}

Recall that $\G_S$ is the hereditary graph class defined by a finite set $S$ of forbidden induced subgraphs. We now characterise when $\G_S$ has bounded pathwidth. \citet{LR2022treewidth} showed that $\G_S$ has bounded treewidth if and only if $S$ includes a complete graph, a complete bipartite graph, a tripod and a semi-tripod. We strengthen this result to show that $\G_S$ in fact has bounded pathwidth if and only if $S$ includes a complete graph, a complete bipartite graph, a tripod and a semi-tripod. 

\FiniteSetPW*

\begin{proof}
    If $S$ does not include a complete graph, a complete bipartite graph, a tripod or a semi-tripod, then by the observations of \citet{LR2022treewidth}, $\G_S$ has unbounded treewidth and thus has unbounded pathwidth.
    
    Now assume $S$ contains a complete graph, a complete bipartite graph, a tripod and a semi-tripod. Observe that for every tripod (semi-tripod), there exists $k\in \N$ such that every (line graph of a) subdivision of $T_k$ contains the tripod (semi-tripod) as an induced subgraph. For the sake of contradiction, suppose $\G_S$ has unbounded pathwidth. By the results of \citet{LR2022treewidth}, there exists $w\in \N$ such that every graph in $\G_S$ has treewidth at most $w$. Thus every graph in $\G_S$ is $K_{w+2}$-minor-free. Since $\G_S$ has unbounded pathwidth \cref{MainTheoremPWMinor}, implies that for every $k\in \N$, there exists a graph $G_k\in \G_S$ that contains a subdivision of $T_k$ or the line graph of a subdivision of $T_k$ as an induced subgraph. Therefore, for $k$ sufficiently large, $G_k$ contains the tripod or the semi-tripod in $S$ as an induced subgraph, a contradiction.
\end{proof}

\section{Conclusion}\label{SectionGeneral}
We raise the following open problem: what are the unavoidable induced subgraphs for graphs with large pathwidth in the general setting? Apart from subdivisions of large complete binary trees and line graphs of subdivisions of large complete binary trees, the other obvious candidates are complete graphs and complete bipartite graphs. While it is tempting to conjecture that this list is exhaustive, we conclude with a non-trivial family of graphs with unbounded pathwidth that avoids these graphs as induced subgraphs which highlights the difficulty of resolving this problem. Let \defn{$\mathcal{S}(T_k)$} be the family of graphs that are subdivisions of $T_k$ or line graphs of subdivisions of $T_k$. The following proof is inspired by a recent construction of Davies (see \cite{ACHS2022inducedV}).

\begin{prop}
For every $k\in \N$, there exists a graph with pathwidth at least $k$ that forbids $\{K_3, K_{2,2}\} \cup \mathcal{S}(T_5)$ as induced subgraphs.
\end{prop}

\begin{proof}
Let ${T}_{2k}^*$ be a subdivision of $(T_{2k},r)$ such that, for the \textsc{bfs}-layering $V_0,V_1,\dots$ of ${T}_{2k}^*$ with $V_0=\{r\}$, if $u,v\in V(T_{2k})$ and $u\in V_i$ and $v\in V_j$, then $|i-j|\geq 3$. Let $\hat{T}_{2k}$ be obtain from $T_{2k}^*$ by adding the edges $uy$ whenever $u\in V(T_{2k})$ and $u,y\in V_a$ for some $a\in \N$. Let $X:=V(T_{2k})\subseteq V(\hat{T}_{2k})$.

We claim that $\hat{T}_{2k}$ has the desired properties. By construction, $\hat{T}_{2k}$ contains a subdivision of $T_{2k}$ as a subgraph and thus has pathwidth at least $k$ \cite{scheffler1989die}. Moreover, since the union of any three layers induces a forest, $\hat{T}_{2k}$ has girth at least $6$ and thus forbids $K_3, K_{2,2}$ as well as the line graph of any subdivision of $T_5$ as induced subgraphs. 

\begin{figure}[ht]
    \centering
    \includegraphics[width=0.7\textwidth]{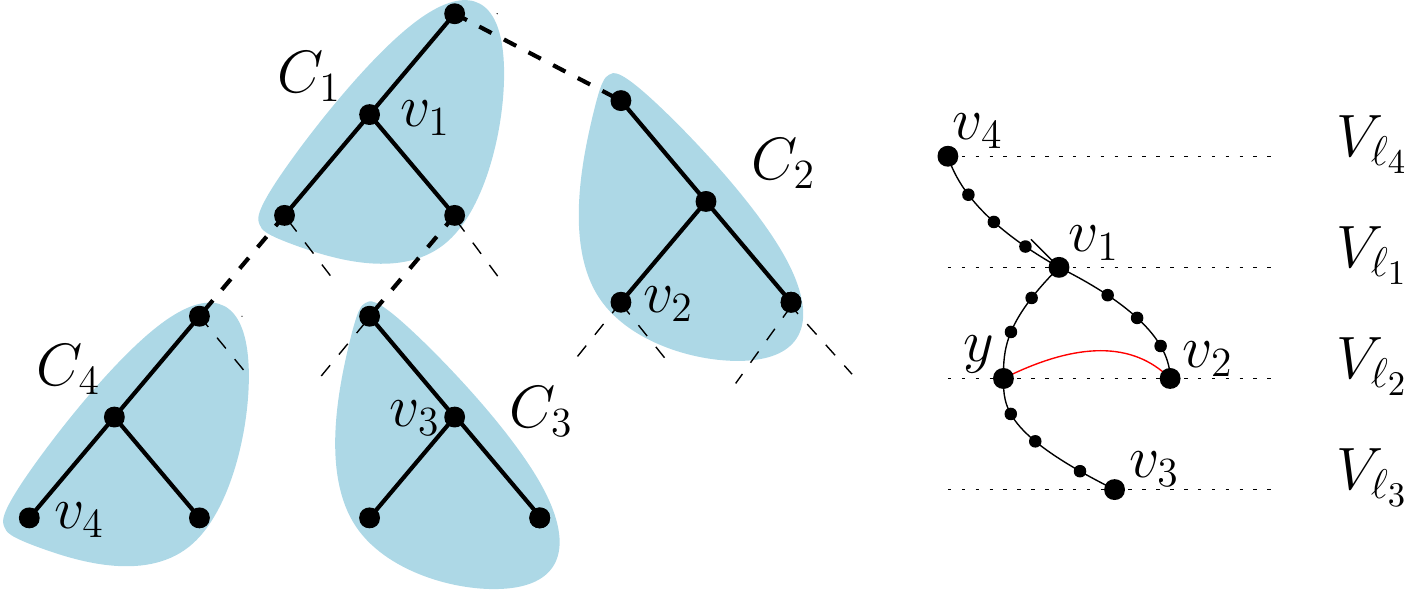}
\caption{Left: four disjoint claws in $T_5'$. Right: non-edge in $T_5'$.}
    \label{fig:T5}
\end{figure}

It remains to show that $\hat{T}_{2k}$ does not contain a subdivision of $T_5$ as an induced subgraph. For the sake of contradiction, suppose $\hat{T}_{2k}$ does contain a subdivision $T_5'$ of $T_5$ as an induced subgraph. Since $\hat{T}_{2k}- X$ has maximum degree $2$, every induced claw in $\hat{T}_{2k}$ contains a vertex from $X$. Observe that $T_5'$ contains four disjoint claws $C_1,C_2,C_3,C_4$ such that there are disjoint, nonadjacent paths from $C_1$ to $C_j$ for each $j\in \{2,3,4\}$ (see \cref{fig:T5}). For each $i \in \{1,2,3,4\}$, let $v_i\in V(C_i)\cap X$ and let $\ell_i$ be such that $v_i\in V_{\ell_i}$. Since $|V_a\cap X|\leq 1$ for all $a \in \N$, it follows that $\ell_i\neq \ell_j$ for all distinct $i,j \in \{1,2,3,4\}$. Thus, by the pigeon-hole principle, we may assume without loss of generality that either $\ell_1<\ell_2<\ell_3$ or $\ell_1>\ell_2>\ell_3$. Since $V_0,V_1,\dots$ is a layering, it follows that the $(v_1,v_3)$-path in $T_5'$ contains an internal vertex $y\in V_{\ell_2}$. However, since $v_2$ dominates $V_{\ell_2}$, it follows that the edge $v_2y$ is present which is a non-edge in $T_5'$ thus contradicting $T_5'$ being an induced subgraph (see \cref{fig:T5}).
\end{proof}

\subsection*{Acknowledgement}
Thanks to David Wood for helpful conversations and comments.

\fontsize{9.5}{10.5} 
\selectfont 
\let\oldthebibliography=\thebibliography
\let\endoldthebibliography=\endthebibliography
\renewenvironment{thebibliography}[1]{%
	\begin{oldthebibliography}{#1}%
		\setlength{\parskip}{0.2ex}%
		\setlength{\itemsep}{0.2ex}%
	}{\end{oldthebibliography}}

\bibliographystyle{DavidNatbibStyle}
\bibliography{main.bbl}	
\end{document}